\documentclass[11pt]{amsart}

\usepackage{tikz-cd}
\usepackage[all, cmtip]{xy}
\usepackage[utf8]{inputenc}

\renewcommand{\phi}{\varphi}
\renewcommand{\epsilon}{\varepsilon}

\newcommand{\PP}{\mathbb{P}}

\newcommand{\gl}{\lambda}

\renewcommand{\models}{\vDash}

\renewcommand{\iff}{\Leftrightarrow}

\usepackage{amscd}

\usepackage{ stmaryrd }
\usepackage{enumerate}

\usepackage{hyperref}
\usepackage{fancyhdr}
\usepackage[english]{babel}
\usepackage[T1]{fontenc}
\usepackage[utf8]{inputenc}

\usepackage{amsmath, amsthm, amscd, amsfonts}
\usepackage{amssymb, tikz}
\usepackage{mathtools}
\usepackage{amscd}

\usepackage{hyperref}
\usepackage[capitalise]{cleveref}

\usepackage{setspace}
\makeatletter
\renewcommand\footnotesize{%
	\@setfontsize\footnotesize\@ixpt{9}%
	\abovedisplayskip 8\p@ \@plus2\p@ \@minus4\p@
	\abovedisplayshortskip \z@ \@plus\p@
	\belowdisplayshortskip 4\p@ \@plus2\p@ \@minus2\p@
	\def\@listi{\leftmargin\leftmargini
		\topsep 4\p@ \@plus2\p@ \@minus2\p@
		\parsep 2\p@ \@plus\p@ \@minus\p@
		\itemsep \parsep}%
	\belowdisplayskip \abovedisplayskip
}
\makeatother

\newcommand{\OB}{\mathrm{OB}}

\renewcommand{\emptyset}{\varnothing}

\renewcommand{\rho}{\varrho}

\newtheorem{thm}{Theorem}[section]

\newtheorem{lemma}[thm]{Lemma}

\newtheorem*{mainthm}{Main Theorem}

\theoremstyle{remark}

\newtheorem{rem}[thm]{Remark}

\renewcommand{\iff}{\Leftrightarrow}

\theoremstyle{definition}
\newtheorem{defn}[thm]{Definition}

\def\MPB{{\mathbb{P}}}
\def\MQB{{\mathbb{Q}}}
\def\MRB{{\mathbb{R}}}

\def\k{\kappa}

\def\l{\lambda}

\def\a{\alpha}

\newcommand{\indep}{\,\,\raise.2em\hbox{$\,\mathrel|\kern-.93em\lower.4em\hbox{$\smile$}$}}
\newcommand{\nindep}{\,\,\raise.2em\hbox{$\mathrel|\kern-.945em\lower.4em\hbox{$\smile$}
		\kern-.75em\hbox{\char'57}$}\;}

\newcommand{\shiftleft}[2]{\makebox[0pt][r]{\makebox[#1][l]{#2}}}

\newcommand{\myind}[1]{   \indep{\raise 7pt\hbox{\shiftleft{4pt}{{\tiny #1}}}}   }
\newcommand{\nmyind}[1]{   \indep{\raise 7pt\hbox{\shiftleft{4pt}{{\tiny #1}}}}{\raise .5pt\hbox{\shiftleft{10pt}{$\not$}}}   }
\newcommand{\myhatind}[1]{
	verset{\sim}{\indep}{\raise 7pt\hbox{\shiftleft{5pt}{{\tiny #1}}}}  }

\newtheoremstyle{InProofNum}
{\topsep}{\topsep}              
{}                      
{}                              
{\itshape}                     
{.}                             
{ }                             
{\thmname{#1}\thmnote{ #3}}
\theoremstyle{InProofNum}

\DeclareMathOperator{\crit}{\mathrm{crit}}
\DeclareMathOperator{\HOD}{\textbf{HOD}}

\DeclareMathOperator{\dom}{\mathrm{dom}}

\DeclareMathOperator{\ZFC}{\textbf{ZFC}}

\DeclareMathOperator{\C}{C}
\DeclareMathOperator{\DTP}{\textbf{DTP}}

\DeclareMathOperator{\Ult}{Ult}

\DeclareMathOperator{\Suc}{Suc}
\DeclareMathOperator{\Lev}{Lev}

\begin{document}

\title{Definable tree property for uncountable regular cardinals}

\author[M. Golshani ]{Mohammad Golshani }

\address{School of Mathematics, Institute for Research in Fundamental Sciences (IPM), P.O. Box:
19395-5746, Tehran, Iran.}

\email{golshani.m@gmail.com}

\urladdr{https://math.ipm.ac.ir/~golshani/}

\author[M. Mirabi]{Mostafa Mirabi}

\address{Department of Mathematics and Computer Science, The Taft School,
Watertown, CT 06795, USA.}

\email{mmirabi@wesleyan.edu}
\urladdr{https://sites.google.com/site/mostafamirabi/}




\thanks{The first author's research has been supported by a grant from IPM (No. 1402030417).}

\maketitle

\begin{abstract}
The primary goal of this paper is to establish a model of $\ZFC$  wherein the definable tree property is affirmed for all uncountable regular cardinals. This endeavor commences with the utilization of both a supercompact cardinal and a measurable cardinal that exceeds it. Subsequently, we construct a $\ZFC$ model. Within this model, we demonstrate that the definable tree property holds for all uncountable regular cardinals.   Thereby we respond to an inquiry raised in \cite{daghighi-pourmahdian}.
\end{abstract}

\section{Introduction}
\label{introduction}

This paper is dedicated to the exploration of Magidor's problem, which has been reformulated for the definable tree property. Recall that, for a cardinal $\kappa$, the \emph{tree property} at $\kappa$ asserts the absence of $\kappa$-Aronszajn trees; where  a $\kappa$-Aronszajn tree is a tree of height $\kappa$ in which all levels have size less than $\kappa$ and all branches have height less than $\kappa$. Magidor's query revolves around the consistency of the tree property being applicable to all regular cardinals greater than $\aleph_1$. Despite the numerous research efforts in this area, it seems we are very far from solving it.  Within this paper, we undertake the redefined version of Magidor's problem, focusing on the definable tree property, and present a comprehensive solution to it.

For a regular  cardinal $\kappa$, let the definable tree property at $\kappa$, denoted $\DTP(\kappa)$, be the assertion ``any
$\kappa$-tree definable in the structure $(H(\kappa), \in)$ has a cofinal branch''.
In \cite{leshem}, it is proved that if $\kappa$ is regular and $\lambda > \kappa$ is a $\Pi^1_1$-reflecting cardinal, then in the generic extension by the Levy collapse
$\mathrm{Col}(\kappa, < \lambda)$, $\DTP(\lambda)$ holds. In \cite{daghighi-pourmahdian}, this result is extended to get the definable tree property at successor of all regular cardinals.

Building upon the findings in \cite{daghighi-pourmahdian} and \cite{leshem}, we proceed to establish a global consistency result. This involves building a model of $\ZFC$ wherein the definable tree property is true for all uncountable regular cardinals. Our result not only resolves an inquiry posed in \cite{daghighi-pourmahdian} but also furnishes a response to the definable version of Magidor's question.

 \begin{mainthm}\label{main theorem}
 Assume $\kappa$ is a supercompact cardinal and $\lambda > \kappa$ is measurable. Then there is a generic extension $W$ of the universe in which the following hold:
\begin{enumerate}
\item  $\kappa$ remains inaccessible.
\item Definable tree property  holds at all uncountable regular cardinals less than $\kappa$.
 \end{enumerate}
In particular, the rank initial segment  $W_\kappa$ of $W$ is a model of $\ZFC$ in which definable tree property  holds at all uncountable regular cardinals.
 \end{mainthm}

Note that we have the flexibility to reduce the large cardinal strength of $\lambda$ to a $\Pi^1_1$-reflecting cardinal.

This paper is organized as follows.  In Section \ref{some preliminaries}, we fix some notation and conventions and present some preliminaries and results that will be used throughout the paper. In Section \ref{proof of main theorem}, we prove the Main Theorem \ref{main theorem}. We assume that the reader is familiar with the papers
\cite{gitik-merimovich} and \cite{merimovich}, though we have provided some background from these papers in the next section.


\section{Preliminaries}
\label{some preliminaries}
The term tree encompasses various meanings in different contexts. In this paper, when we say a tree, we are referring a partially ordered set  $(T, <_T)$ such that for any element $t\in T$, the set $\left\{s \in T : s <_T t \right\}$, which consists of the predecessors of $t$, must be well-ordered under the relation $<_T$. Additionally, there exists a root element $r\in T$ such that, for every element $t$ distinct from $r$, $r<_T t$.

First we recall some basic definition, observations and facts that will serve as building blocks for the rest of the paper.

Most of the following concepts and definitions can be found in \cite{merimovich}.

   We begin with the assumption of $GCH$ in conjunction with a Mitchell increasing sequence of extenders denoted as $\bar{E}=\langle E_\xi: \xi < o(\bar{E}) \rangle$ defined on $\kappa,$ and $\lambda$ represents the least inaccessible cardinal greater than $\kappa.$ We define ``\textit{supercompact extender based Radin forcing}'', denoted by $\PP_{\bar{E}}$, which was introduced by Merimovich \cite{mer5}, and characterized by the following key properties:

$\PP_{\bar{E}}$ preserves the inaccessibility of $\kappa$.
It collapses all cardinals in the interval $(\kappa, \lambda)$ while preserving cardinals greater than or equal to $\lambda,$ ensuring that $\kappa^+=\lambda$ in the extension $\PP_{\bar{E}}$.
It adds a club $C=\left\{\kappa_\xi: \xi<\kappa\right\}$ consisting of $V$-measurable cardinals.
For limit ordinal $\xi<\kappa$ and the least inaccessible cardinal $\lambda_\xi$ above $\kappa_\xi,$ if $\mu$ is a regular cardinal in the interval $(\kappa_\xi, \lambda_\xi),$ then there exists a cofinal sequence into $\mu$ with an order type less than or equal to $\kappa_\xi.$ This implies that all cardinals $\mu$ in the interval $(\kappa_\xi, \lambda_\xi)$ are collapsed.
The forcing preserves $\lambda_\xi,$ meaning that for limit ordinals $\xi<\kappa,$ $\kappa_\xi^+=\lambda_\xi$ in the extension by $\PP_{\bar{E}}$. All other cardinals below $\kappa$ remain preserved.

We proceed to consider $G$ as $\PP_{\bar{E}}$-generic over $V$ and introduces a new forcing notion $\PP_{\bar{E}}^\pi$, termed the "projected supercompact extender based Radin forcing," along with a Prikry-type projection $\pi: \PP_{\bar{E}} \rightarrow \PP_{\bar{E}}^\pi.$ The forcing notion $\PP_{\bar{E}}^\pi$ adds the club $C$ without collapsing any cardinals. Importantly, the quotient forcing exhibits sufficient homogeneity properties to establish that $HOD^{V[G]} \subseteq V[G^\pi],$ where $G^\pi$ is the filter generated by $\pi[G].$ This leads to the conclusion that, for all limit ordinals $\xi<\kappa,$ $(\kappa_\xi^+)^{V[G]}=\lambda_\xi$ remains inaccessible in $HOD^{V[G]}$.

We proceed with recalling the "\textit{supercompact extender based Radin forcing}." This particular forcing was initially defined by Merimovich \cite{merimovich}  and serves as the foundation for the subsequent proof presented in the paper. The section outlines the essential properties of this forcing.

Assume that the $GCH$ is valid in the ground model, and let a supercompact cardinal $\k$ with $\lambda$ as the smallest strongly inaccessible cardinal greater than $\k$ be given.
Let  $j: V \rightarrow M$ be an elementary embedding with  $\crit(j)=\k$, and $j(\k)>\lambda$ such that  $M\supseteq ^{<\lambda}M$.
For each ordinal $\a$, $\lambda_\alpha$ is defined as the minimal $\eta$ such that $\a < j(\eta)$. The generators of this embedding $j$, denoted by $\mathfrak{g}(j)=\left\{\k_\xi : \xi\in \textbf{Ord}\right\}$, are recursively defined as:

\(\kappa_\xi = \min\left\{\alpha \in \textbf{Ord} : (\forall \zeta < \xi) \, (\forall \eta \in \textbf{Ord}) \, (\forall f : \eta \rightarrow \textbf{Ord}) \, (j(f)(\kappa_\zeta) \neq \alpha)\right\}.\)

If $\mathfrak{g}(j)$ forms a set, we can encode $j$ using an
extender $E = \langle E(\alpha): \a \in \mathfrak{g}(j) \rangle$, where $E(\a)$ is a measure on $\lambda_\alpha$, for each $\alpha\in \mathfrak{g}(j)$, which is defined by
$$ \forall A\subseteq \lambda_\alpha\, (A\in E(\alpha))\Leftrightarrow \alpha\in j(A).$$

We deal with embeddings having their generators below $j(\lambda)$ (and hence a set), and consider the natural elementary embedding $j_E: V \rightarrow Ult(V, E).$ We may further assume that $j=j_E.$  We focus on embeddings where their generators are below $j(\lambda)$, which implies they form a set. We also consider the natural elementary embedding $j_E: V \rightarrow \text{Ult}(V, E)$. For simplicity, we assume $j$ coincides with $j_E$.

Consider a sequence of extenders $\bar{E}=\langle E_\xi: \xi < o(\bar{E})\rangle$ on $\k$ satisfying the following conditions:
\begin{enumerate}
    \item $\bar{E}$ is Mitchell increasing, meaning that for each $\xi< o(\bar{E})$, we have $\langle E_\zeta: \zeta < \xi \rangle \in E_\xi$.

    \item For the corresponding elementary embeddings $j_{E_\xi}: V \rightarrow M_\xi \simeq \text{Ult}(V, E_\xi)$, the following hold:

    \begin{enumerate}
        \item[(2-1)] The $\crit(j_{E_\xi})=\k$, and $j_{E_\xi}(\k) \geq \lambda$.

        \item[(2-2)] $\lambda$ is the minimal ordinal such that $M_\xi \nsupseteq ^{\lambda}M_\xi$, and therefore, $M_\xi \supseteq ^{<\lambda}M_\xi$.

        \item[(2-3)] For each $\xi < o(\bar{E})$, the generator of $j_{E_\xi}$,  $\mathfrak{g}(j_{E_\xi})$, is contained within $\sup j_{E_\xi}[\lambda]$.
    \end{enumerate}
\end{enumerate}
It's worth noting that if $\xi_1 < \xi_2 < o(\bar{E})$, then $j_{E_{\xi_1}}(\lambda) < j_{E_{\xi_2}}(\lambda)$. Additionally, let $\lambda \leq \epsilon \leq \sup\left\{j_{E_\xi}(\lambda): \xi < o(\bar{E})\right\}$.

An extender sequence $\bar{\nu}$ is structured as $ \langle \tau, e_0, \dots, e_\xi, \dots  \rangle_{\xi<\mu}$. Here, $\langle e_\xi: \xi < \mu \rangle$ is a Mitchell increasing sequence of extenders, all sharing the same critical points and closure points. Furthermore, we have $crit(e_0) \leq \tau < j_{e_0}(\alpha),$ where $\alpha$ is the closure point of $M_{e_0}$.
The order of the extender sequence $\bar{\nu}$ is $\mu$ which we represent as $o(\bar{\nu}) = \mu.$ Specifically, we use $\bar{\nu}_0$ to denote $\tau$ and naturally extend this notation to $\bar{\nu}_{1+\xi}$ for $e_\xi.$
It's worth noting that formally, the Mitchell order function $o(...)$ is defined for different types of objects. The first type of object takes the form $\langle E_\xi: \xi<\mu \rangle,$ while the second type is represented as $\langle \tau, e_0, \dots, e_\xi, \dots  \rangle_{\xi<\mu}.$ However, in both cases, only the extenders are taken into consideration, eliminating any potential for confusion.

The set $\mathfrak{D}$ serves as the base set in the domain of functions. For every $\k \leq \a < \sup\left\{j_{E_\xi}(\lambda): \xi < o(\bar{E})\right\}, \a \notin j[\lambda]$, we establish the definition:

\begin{center}
$\bar{\a} = \langle \a \rangle^\frown \langle E_\zeta: \zeta < o(\bar{E}), \a < j_{E_\zeta}(\lambda) \rangle.$
\end{center}

Following this, we further define:

\begin{center}
$\mathfrak{D} = \left\{\bar{\a}: \k \leq \a < \epsilon\right\}.$
\end{center}

The order $<$ on $\mathfrak{D}$ is defined as $\bar{\a} < \bar{\beta}$ if and only if $\a < \beta$.

Subsequently, the set $\mathfrak{R}$ serves as the base set for the range of functions:

\begin{center}
$\mathfrak{R} = \left\{ \bar{\nu}\in V_\lambda: \bar{\nu} \text{ is an extender sequence }  \right\}.$
\end{center}

On $\mathfrak{R}$, the order $<$ is defined by $\bar{\nu} < \bar{\mu}$ if and only if $\bar{\nu}_0 < \bar{\mu}_0$.
For technical reasons we assume that  that $\langle \rangle \in \mathfrak{R}$.

Let $d\in [j(\lambda)]^{<\lambda}$ be such that $\k, |d| \in d.$ Then we define $\text{OB(d)}$ as follows. $\nu \in \text{OB(d)}$ if and only if:
\begin{enumerate}
\item[$(\divideontimes)_1$] $\nu: \dom(\nu) \rightarrow \lambda,$ where $\dom(\nu) \subseteq d,$
\item[$(\divideontimes)_2$] $\k, |d| \in \dom(\nu),$
\item[$(\divideontimes)_3$] $|\nu| \leq \nu(|d|),$
\item[$(\divideontimes)_4$] $\forall \a<\l$ $(j(\a)\in \dom(\nu) \Rightarrow \nu(j(\a))=\a ),$
\item[$(\divideontimes)_5$] $\a\in \dom(\nu) \Rightarrow \nu(\a) < \gl_\a,$
\item[$(\divideontimes)_6$] $\a < \beta$ in $\dom(\nu) \Rightarrow \nu(\a) < \nu(\beta).$
\end{enumerate}
Also for $\nu_0, \nu_1 \in \text{OB(d)},$ set $\nu_0 < \nu_1$ iff
\begin{enumerate}
\item[$(\divideontimes)_7$]$\dom(\nu_0) \subseteq \dom(\nu_1),$

\item[$(\divideontimes)_8$] For all $\a\in \dom(\nu_0)\setminus j[\gl], \nu_0(\a) < \nu_1(\a).$
\end{enumerate}

We now define the forcing notion $\PP^*_E.$

$\PP^*_E$ consists of all functions $f: d \rightarrow \lambda^{<\omega}$, where  $d\in [j(\lambda)]^{<\lambda}$, $\k, |d| \in d,$ and  such that
\begin{enumerate}
    \item For any $j(\a) \in d, f(j(\a))=\langle \a \rangle,$
    \item For any $\a\in d\setminus j[\gl], $ there is some $k<\omega$ such that
\begin{center}
$f(\a)= \langle  f_0(\a), \dots, f_{k-1}(\a)               \rangle \subseteq \gl_\a $
\end{center}
is a finite increasing subsequence of $\gl_\a.$ For $f, g\in \PP^*_E,$
\begin{center}
$f \leq^*_{\PP^*_E} g \Leftrightarrow f \supseteq g.$
\end{center}
\end{enumerate}

\begin{rem}
$\langle \PP^*_E, \leq^*_{\PP^*_E} \rangle \approx Add(\lambda, |j(\lambda)|).$
\end{rem}

Note that for $d \in P_\lambda(\mathfrak{D})$, $\nu \in \OB(d)$ if and only if:

\begin{enumerate}
    \item[$(\clubsuit)$1] $\nu: \dom(\nu) \rightarrow \mathfrak{R},$
    \item[$(\clubsuit)$2] $\bar{\k} \in \dom(\nu) \subseteq d \cup j[\lambda],$
    \item[$(\clubsuit)$3] $|\nu| < \nu(\bar{\k})_0,$
    \item[$(\clubsuit)$4] $\forall \a < \lambda \, (j(\a) \in \dom(\nu) \Rightarrow \nu(j(\a)) = \langle \a \rangle),$
    \item[$(\clubsuit)$5] $\forall \bar{\a} \in \dom(\nu)\setminus j[\lambda] \, (o(\nu(\bar{\a})) < o(\bar{\a})),$
    \item[$(\clubsuit)$6] For each $\bar{\a} \in \dom(\nu)\setminus j[\lambda]$ such that $\bar{\a} \neq \bar{\k}$, the following is satisfied: let
        \begin{center}
          $\nu(\bar{\k}) = \langle \tau, e_0, \dots, e_\xi, \dots \rangle_{\xi < \zeta_\k}$ (where $crit(e_0) = \tau$)
         \end{center}
          and
         \begin{center}
              $\nu(\bar{\a}) = \langle \tau', e'_0, \dots, e'_\xi, \dots \rangle_{\xi < \zeta_\a}$,
         \end{center}
          then $\langle e_{\zeta+\xi}: \xi < \zeta_\k \rangle = \langle e'_{\xi}: \xi < \zeta_\a \rangle$, where $\zeta < \zeta_\k$ is minimal such that $\tau' \in (\sup_{\zeta'<\zeta}j_{e_{\zeta'}}(\tau), j_{e_\zeta}(\sigma))$, where $\sigma$ is the closure point of $e_\zeta.$
    \item[$(\clubsuit)$7] $\forall \bar{\a}, \bar{\beta} \in \left(\dom(\nu)\setminus j[\lambda]\right) \, \bigg(\bar{\a} < \bar{\beta} \Rightarrow \nu(\bar{\a}) < \nu(\bar{\beta})\bigg).$
\end{enumerate}

On $\OB(d)$, the partial order $<$ is defined by $\mu < \nu$ if either

\begin{center}
$\forall \bar{\a} \in \dom(\mu) \cap \dom(\nu) \, (o(\mu(\bar{\a})) > o(\nu(\bar{\a}))$ and $\mu(\bar{\a}) < \nu(\bar{\a}))$
\end{center}
or
\begin{center}
$\dom(\mu) \subseteq \dom(\nu)$ and $\forall \bar{\a} \in \dom(\mu) \, (o(\mu(\bar{\a})) \leq o(\nu(\bar{\a}))$ and $\mu(\bar{\a}) < \nu(\bar{\a})).$
\end{center}

 Moreover, assuming  $d\in P_\gl(\mathfrak{D})$, we have
\begin{enumerate}
\item[($\spadesuit$)1]  If $T \subseteq \OB(d)^{<\xi} (1<\xi \leq \omega)$ and  $n<\omega.$ Then
\begin{itemize}
\item $Lev_n(T)=T \cap \OB(d)^{n+1},$

\item $\Suc_T(\langle \rangle) = \Lev_0(T),$

\item $\Suc_T(\langle \nu_o, \dots, \nu_{n-1}         \rangle)=\left\{\mu\in \OB(d): \langle \nu_o, \dots, \nu_{n-1}, \mu \rangle \in T    \right\}.$
\end{itemize}
\item[($\spadesuit$)2] For $\langle \nu \rangle \in T,$ let
\begin{center}
 $T_{\langle \nu \rangle}= \left\{  \langle \nu_o, \dots, \nu_{k-1} \rangle: k<\omega, \langle \nu, \nu_o, \dots, \nu_{k-1} \rangle \in T \right\}$
\end{center}
 and define by recursion for $\langle \nu_o, \dots, \nu_{n-1} \rangle \in T,$
\begin{center}
$T_{\langle \nu_o, \dots, \nu_{n-1} \rangle}= (T_{\langle \nu_o, \dots, \nu_{n-2} \rangle})_{\langle \nu_{n-1} \rangle}.$
\end{center}
\item[($\spadesuit$)3] The measures $E_\xi(d)$ $(\xi < o(\bar{E}))$ on $\OB(d)$ are defined as follows:
\begin{center}
$\forall X\subseteq \OB(d)$ $($ $X\in E_\xi(d) \Leftrightarrow mc_\xi(d)\in j_{E_\xi}(X) $ $)$,
\end{center}
where
\begin{center}
$mc_\xi(d)=\left\{ \langle j_{E_\xi}(\bar{\a}), R_\xi(\bar{\a}) \rangle : \bar{\a}\in d, \a < j_{E_\xi}(\l) \right\},$
\end{center}
and $R_\xi$ is defined for each $\k \leq \a < \epsilon$ by
\begin{center}
$R_\xi(\a)=\langle \a \rangle ^{\frown} \langle E_{\xi'}: \xi' < \xi,    \a < j_{E_{\xi'}}(\l)                     \rangle.$
\end{center}
Also set
\begin{center}
$E(d)=\bigcap\left\{ E_\xi(d): \xi < o(\bar{E})            \right\}.$
\end{center}
\item[($\spadesuit$)4] A tree $T \subseteq \OB(d)^{<\omega}$ is called a $d$-tree, if
\begin{itemize}
\item For each $\langle \nu_o, \dots, \nu_{n-1} \rangle\in T,$ we have $\nu_0 < \dots < \nu_{n-1},$
\item $\forall \langle \nu_o, \dots, \nu_{n-1} \rangle \in T, \Suc_T(\langle \nu_o, \dots, \nu_{n-1} \rangle)\in E(d).$
\end{itemize}
\item[($\spadesuit$)5] Assume $c\in P_{\k^+}(\mathfrak{D}), c \subseteq d,$ and $T$ is a tree with elements from $\OB(d)$. Then the projection of $T$ to a tree with elements from $\OB(c)$ is
\begin{center}
$T\upharpoonright c = \left\{\langle \nu_o\upharpoonright c, \dots, \nu_{n-1}\upharpoonright c \rangle: n<\omega, \langle \nu_o, \dots, \nu_{n-1} \rangle\in T                    \right\}$.
\end{center}
\end{enumerate}


Recall that $\PP^*_{\bar{E}, \epsilon}$ consists of all functions $f: d \rightarrow ^{<\omega}\mathfrak{R}$ such that
\begin{enumerate}
\item[$(\bigstar)1$] $\bar{\k}\in d\in P_\l(\mathfrak{D})$,
\item[$(\bigstar)2$] For each $\bar{\a}\in d, f(\bar{\a})=\langle f_0(\bar{\a}), \dots, f_{k-1}(\bar{\a}) \rangle$ is an increasing sequence in $\mathfrak{R},$
\item[$(\bigstar)3$] For each $\bar{\a}\in d$ and $i< |f(\bar{\a})|,$ $(o(f_i(\bar{\a})) < o(\bar{\a}) ),$
\item[$(\bigstar)4$] For each $\bar{\a}\in d,$ the sequence $\langle  o(f_i(\bar{\a})): i< |f(\bar{\a})|       \rangle$ is non-increasing.
\end{enumerate}

For $f, g \in \PP^*_{\bar{E}, \epsilon},$ we say $f$ is an extension of $g$ $(f \leq^*_{\PP^*_{\bar{E}, \epsilon}} g)$ if $f \supseteq g.$

\begin{rem}
Clearly $\PP^*_{\bar{E}, \epsilon} \simeq Add(\gl, \epsilon).$
\end{rem}

We write $OB(f), E_\xi(f), E(f), mc_\xi(f)$ and $f$-tree, for $OB(\dom(f)), E_\xi(\dom(f)),$
$E(\dom(f)),$
\\
$mc_\xi(\dom(f))$ and $\dom(f)$-tree respectively, where $f\in \PP^*_{\bar{E}, \epsilon}.$
The following Lemma is immediate.

\begin{lemma}
$\langle \PP^*_{\bar{E}, \epsilon},     \leq^*_{\PP^*_{\bar{E}, \epsilon}}       \rangle$ satisfies the $\gl^+-c.c.$
\end{lemma}

Assume $f\in \PP^*_{\bar{E}, \epsilon}$ and $\nu\in OB(f).$ Define $g=f_{\langle \nu \rangle}$ to be of the form $g= g_{\leftarrow}^{\frown}g_{\rightarrow}$ (the case $g_{\leftarrow}=\emptyset$ is allowed) where
\begin{enumerate}
\item $\dom(g_{\rightarrow})=\dom(f),$
\item For each $\bar{\a}\in \dom(g_{\rightarrow})$,
\begin{center}
 $g_{\rightarrow}(\bar{\a}) = \left\{ \begin{array}{l}
       f(\bar{\a}) \upharpoonright k^{\frown} \langle \nu(\bar{\a}) \rangle  \hspace{1.1cm} \text{ if } \bar{\a}\in \dom(\nu), \nu(\bar{\a}) > f_{|f(\bar{\a})|-1}(\bar{\a}),\\
       f(\bar{\a})  \hspace{3cm} \text{}  Otherwise,
 \end{array} \right.$
\end{center}
where
\begin{center}
$k=\min\left\{l \leq |f(\bar{\a})|: \forall l \leq i < |f(\bar{\a})|, o(f_i(\bar{\a})) < o(\nu(\bar{\a}))        \right\}.$
\end{center}
The above value of $k$ is defined so as to ensure that $\langle o(f_i(\bar{\a})) : i<k \rangle ^{\frown} \langle   o(\nu(\bar{\a}))   \rangle$ is non-increasing.
\item $\dom(g_{\leftarrow})=\left\{\nu(\bar{\a}): \bar{\a}\in \dom(\nu),  o(\nu(\bar{\a})) >0         \right\}$,
\item For each $\bar{\a}\in \dom(\nu)$ with $o(\nu(\bar{\a})) >0$ we have
\begin{center}
$g_{\leftarrow}(\nu(\bar{\a}))= f(\bar{\a}) \upharpoonright (|f(\bar{\a})|\setminus k),$
\end{center}
where $k$ is defined as above.
\end{enumerate}

Also recall that $\PP_{\bar{E}, \epsilon, \rightarrow}$ consists of pairs $p= \langle f, A \rangle$ where
\begin{enumerate}
\item $f\in \PP^*_{\bar{E}, \epsilon},$
\item $A$ is an $f$-tree such that for each $\langle \nu \rangle \in A$ and each $\bar{\a}\in \dom(\nu)$
\begin{center}
$f_{|f(\bar{\a})|-1}(\bar{\a}) < \nu(\bar{\a}).$
\end{center}
\end{enumerate}
In this case we write $f^p, A^p$ and $mc_\xi(p)$ for $f, A$ and $mc_\xi(f)$, respectively.

Suppose $p, q \in \PP_{\bar{E}, \epsilon, \rightarrow}.$ We say $p \leq^*_{\PP_{\bar{E}, \epsilon, \rightarrow}} q$ ($p$ is a Prikry extension of $q$) if
\begin{enumerate}
\item $f^p \leq^*_{\PP^*_{\bar{E}, \epsilon}} f^q,$
\item $A^p \upharpoonright \dom(f^q) \subseteq A^q$.
\end{enumerate}

Suppose $\langle e_i: i< n \rangle$ $(n<\omega)$ is a sequence of extenders such that $e_i \in V_{crit(e_{i+1})}.$  Roughly speaking, we define the product forcing notion $\PP=\prod_{i<n}\PP_{e_i}$ using the definitions of the Prikry with extenders forcing notions coordinatewise. More precisely, for each $\langle p_i: i< n \rangle, \langle q_i: i< n \rangle \in \PP,$
\begin{center}
$\langle p_i: i< n \rangle \leq_\PP \langle q_i: i< n \rangle \Leftrightarrow \forall i<n, p_i \leq_{\PP_{e_1}} q_i,$
\end{center}
and
\begin{center}
$\langle p_i: i< n \rangle \leq^*_\PP \langle q_i: i< n \rangle \Leftrightarrow \forall i<n, p_i \leq^*_{\PP_{e_1}} q_i.$
\end{center}
For $p=\langle p_i: i< n \rangle\in \PP,$ we write $p_{\leftarrow}$ and $p_{\rightarrow}$ to denote $p_0^{\frown} \dots ^{\frown} p_{n-2}$ and $p_{n-1}$ respectively. Also  assuming  $\langle \nu  \rangle \in A^{p_{\rightarrow}}$, we define the condition $p_{\langle \nu \rangle}$ recursively as follows:
\begin{center}
$p_{\langle \nu \rangle}=p_{\leftarrow}^{\frown}p_{\rightarrow \langle \nu \rangle}.$
\end{center}
Note that based on the above definitions of $p_{\leftarrow}$ and $p_{\rightarrow}$, for each $p, q \in \PP$, we have
\begin{center}
$p\leq q \Leftrightarrow (p_{\leftarrow} \leq q_{\leftarrow}$ and $ p_{\rightarrow} \leq q_{\rightarrow} \leq),$
\end{center}
and
\begin{center}
$p\leq^* q \Leftrightarrow (p_{\leftarrow} \leq^* q_{\leftarrow}$ and $ p_{\rightarrow} \leq^* q_{\rightarrow} \leq).$
\end{center}

A well-known fact is that $\langle \PP, \leq_{\PP}, \leq^*_{\PP} \rangle$ constitutes a  Prikry type of forcing, provided that all the forcing notions $\langle \PP_{e_i}, \leq_{\PP_{e_i}}, \leq^*_{\PP_{e_i}} \rangle$ are also of the Prikry type. Notice that since the extenders $e_i$ are distinct, we can easily factor $\PP$ into its constituent parts $\PP_{e_i}$'s. Consequently, a generic extension via $\PP$ can be comprehensively understood by examining the generic extensions via each component $\PP_{e_i}$. With this groundwork, we are now prepared to formally define the forcing notion $\PP_{\bar{E}, \epsilon}$.
A condition $p$ in the forcing notion $\PP_{\bar{E}, \epsilon}$ is of the form $p_{\leftarrow}^{\frown} p_{\rightarrow}$ where
\begin{enumerate}
\item $p_{\rightarrow} \in \PP_{\bar{E}, \epsilon, \rightarrow},$
\item $p_{\leftarrow} \in \prod_{i<n} \PP_{\bar{e}_i}$ $(n<\omega)$, in which $\bar{e}_i$ are extender sequences such that
\begin{enumerate}
\item [(2-1)] $o(\bar{e}_i) \leq o(\bar{E}),$
\item [(2-2)] $\bar{e}_i \in V_{crit(\bar{e}_{i+1})},$
\item [(2-3)] $\langle \nu \rangle \in A^{p_{\rightarrow}} \Rightarrow \nu(\bar{\k})_0 > crit(\bar{e}_{n-1}).$
\end{enumerate}
\end{enumerate}

Also, conditions in $\PP_{\bar{E}, \epsilon}$ have lower parts $\PP_{\bar{E}, \epsilon, \leftarrow}$ which is defined as follows
\begin{center}
$\PP_{\bar{E}, \epsilon, \leftarrow}=\{p_{\leftarrow}: p\in \PP_{\bar{E}, \epsilon}  \}.$
\end{center}
Also for $p\in \PP_{\bar{E}, \epsilon}$, we define $f^p$ recursively to be $f^{p_{\leftarrow}} {}^\frown  f^{p_{\rightarrow}},$ and we denote $f^{p_{\leftarrow}}$ and $f^{p_{\rightarrow}}$ by $f^p_{\leftarrow}$ and $f^p_{\rightarrow}$ respectively.

Let $p, q\in \PP_{\bar{E}, \epsilon}.$ Then $p \leq^*_{\PP_{\bar{E}, \epsilon}} q$ ($p$ is a Prikry extension of $q$) if:
\begin{enumerate}
\item $p_{\rightarrow} \leq^* q_{\rightarrow},$
\item $p_{\leftarrow} \leq^* q_{\leftarrow}.$
\end{enumerate}

\begin{enumerate}
\item[(*1)] Suppose $\mu, \nu \in OB(f)$ are such that $\mu < \nu,$ and for each $\bar{\a}\in \dom(\mu), o(\mu(\bar{\a})) < o(\nu(\bar{\a})).$ Then $\mu \downarrow \nu$, the reflection of $\mu$ by $\nu$ is defined as follows
\begin{center}
$\dom(\mu \downarrow \nu)=\{ \nu(\bar{\a}): \bar{\a}\in \dom(\mu)     \}$
\end{center}
and for each $\bar{\a}\in \dom(\mu)$
\begin{center}
$(\mu \downarrow \nu)(\nu(\bar{\a}))=\mu(\bar{\a}).$
\end{center}
If $\mu_0, \dots, \mu_n, \nu \in OB(f)$ are such that $\mu_i < \nu$ and for each  $\bar{\a}\in \dom(\mu_i), o(\mu_i(\bar{\a})) < o(\nu(\bar{\a})),$ then
$\langle \mu_0, \dots, \mu_n \rangle \downarrow \nu,$ the reflection of $\langle \mu_0, \dots, \mu_n \rangle$ by $\nu$ is defined to be
 $\langle  \mu_0 \downarrow \nu, \dots, \mu_n \downarrow \nu   \rangle.$

\item[(*2)] Suppose $A$ is an $f$-tree and $\langle \nu \rangle \in A.$ The tree $A \downarrow \nu$  consists of all $\langle \mu_0, \dots, \mu_n \rangle \downarrow \nu$ where:
\begin{enumerate}
\item $n<\omega,$
\item $\langle \mu_0, \dots, \mu_n \rangle \in A,$
\item $\forall i\leq n$ $( \mu_i < \nu, \dom(\mu_i) \subseteq \dom(\nu)$ and $\forall \bar{\a}\in \dom(\mu_i), o(\mu_i(\bar{\a})) < o(\nu(\bar{\a})) ).$
\end{enumerate}
\end{enumerate}

It is easily seen that
\begin{center}
$\{\langle \nu \rangle \in A: A \downarrow \nu$ is an $f_{\langle \nu \rangle \leftarrow}$-tree$ \} \in E_1(f),$
\end{center}
and if we consider $\emptyset$ to be an $\emptyset$-tree, then
\begin{center}
$\{\langle \nu \rangle \in A: A \downarrow \nu$ is an $f_{\langle \nu \rangle \leftarrow}$-tree$ \} \in E(f).$
\end{center}

Assume $q\in \PP_{\bar{E}, \epsilon, \rightarrow}$ and $\langle \nu \rangle \in A^q.$ The condition $p\in \PP_{\bar{E}, \epsilon}$ is the one point extension of $q$ by $\langle \nu \rangle$ ($p=q_{\langle \nu \rangle}$) if it is of the form $p_{\leftarrow} {}^{\frown}p_{\rightarrow}$ where $p_{\leftarrow} \in \PP_{\bar{e}, \rightarrow}$
and $p_{\rightarrow} \in \PP_{\bar{E}, \rightarrow}$ are defined as follows:
\begin{enumerate}
\item $f^p=f^q_{\langle \nu \rangle}$,
\item $A^{p_{\rightarrow}} =A^q_{\langle \nu \rangle},$
\item $A^{p_{\leftarrow}} =A^q \downarrow \nu.$
\end{enumerate}
Define $q_{\langle \nu_0, \dots, \nu_n \rangle}$ recursively by
\begin{center}
$q_{\langle \nu_0, \dots, \nu_n \rangle}=(q_{\langle \nu_0, \dots, \nu_{n-1} \rangle})_{\langle \nu_n \rangle}$,
\end{center}
where $\langle \nu_0, \dots, \nu_{n-1} \rangle \in A^q.$

Assume $p\in \PP_{\bar{E}, \epsilon}$ and $\langle \nu \rangle \in A^{p_{\rightarrow}}.$ Then
\begin{center}
$p_{\langle \nu \rangle}=p_{\leftarrow} {}^{\frown} p_{\rightarrow}\langle \nu \rangle.$
\end{center}
Define $p_{\langle \nu_0, \dots, \nu_n \rangle}$ recursively by
\begin{center}
$p_{\langle \nu_0, \dots, \nu_n \rangle}=(p_{\langle \nu_0, \dots, \nu_{n-1} \rangle})_{\langle \nu_n \rangle}$
\end{center}

Let $p, q \in \PP_{\bar{E}, \epsilon}.$ Then $p \leq_{\PP_{\bar{E}, \epsilon}} q$ ($p$ is stronger than $q$) if $p= r^{\frown} s$
and there is $\langle \nu_0, \dots, \nu_{n-1} \rangle \in A^{q_{\rightarrow}}$ such that
\begin{enumerate}
\item $s \leq^*_{\PP_{\bar{E}, \epsilon}} q_{\rightarrow \langle \nu_0, \dots, \nu_{n-1} \rangle},$
\item $r \leq q_{\leftarrow}$.
\end{enumerate}

Assume $p \in \PP_{\bar{E}, \epsilon}.$ Then $p_{\rightarrow} \in \PP_{\bar{E}, \epsilon},$ and we define
\begin{center}
$\PP_{\bar{E}, \epsilon} \downarrow p_{\rightarrow} =\{q\in \PP_{\bar{E}, \epsilon}: q \leq p_{\rightarrow}         \}$
\end{center}
and
\begin{center}
$\PP_{\bar{E}, \epsilon} \downarrow p_{\leftarrow} =\{r: r \leq p_{\leftarrow}  \}.$
\end{center}

Let's state the main properties of our forcing notion.
\begin{lemma}
$\PP_{\bar{E}, \epsilon}$ satisfies the $\l^+$-c.c.
\end{lemma}

\begin{proof}
Assume not, and let $A \subseteq \PP_{\bar{E}, \epsilon}$ be an antichain of size $\l^+.$ We may assume without loss of generality that all $p_{\leftarrow},$ for $p\in A$
are the same (as there are only $\lambda$-many such $p_{\leftarrow}$). Note that for any $p, q \in \PP_{\bar{E}, \epsilon},$ is $f^p$ is compatible with $f^q$ in $\PP^*_{\bar{E}, \epsilon},$ then $p$ and $q$ are compatible in $\PP_{\bar{E}, \epsilon}.$ It follows that $\{f^p: p\in A   \} \subseteq \PP^*_{\bar{E}, \epsilon}$
is an antichain of size $\gl^+,$ which contradicts Lemma 4.17.
\end{proof}

The following factorization property is immediate.
\begin{lemma} (Factorization Lemma)
For any $p\in \PP_{\bar{E}, \epsilon},$
\begin{center}
$\PP_{\bar{E}, \epsilon} \downarrow p \simeq \PP_{\bar{E}, \epsilon} \downarrow p_{\leftarrow} \times \PP_{\bar{E}, \epsilon} \downarrow p_{\rightarrow}.$
\end{center}
\end{lemma}

The next Lemmas are proved in \cite{gitik-merimovich}.
\begin{lemma}
 $\langle  \PP_{\bar{E}, \epsilon}, \leq, \leq^* \rangle$ satisfies the Prikry property.
\end{lemma}

\begin{lemma}
In a $ \PP_{\bar{E}, \epsilon}$-generic extension, $\l$ is preserved.
\end{lemma}

Let $G$ be $\PP_{\bar{E}, \epsilon}$-generic over $V$, and for $\k \leq \a < \epsilon$ let
\begin{center}
$G^{\bar{\a}}=\bigcup \{ f^p_{\rightarrow}(\bar{\a}): p\in G, \bar{\a}\in \dom(f^p_{\rightarrow})          \}$
\end{center}
and
\begin{center}
$C^{\bar{\a}}=\bigcup \{ \bar{\nu}_0: \bar{\nu}\in G^{\bar{\a}}         \}.$
\end{center}

\begin{lemma}
\begin{enumerate}
\item  $C^{\bar{\k}}$ is a club of $\k,$
\item $\a \neq \beta \Rightarrow C^{\bar{\a}} \neq C^{\bar{\beta}},$
\item  Forcing with $\PP_{\bar{E}, \epsilon}$ collapses all cardinal in $(\k, \l)$ onto $\k.$
\end{enumerate}
\end{lemma}

It follows from our results that
\begin{center}
$\l=(\k^+)^{V[G]}.$
\end{center}
Also by Lemma 4.27, $2^\k \geq |\epsilon|,$ and using the $\gl$-chain condition of the forcing, we can conclude that
$(2^\k)^{V[G]} \leq (|\PP_{\bar{E}, \epsilon}|^{<\gl})^{\k}=\epsilon,$ and hence
\begin{center}
 $V[G]\models$``$2^\k=|\epsilon|$''.
\end{center}

\begin{lemma}
 If $cf(o(\bar{E})) > |\epsilon|,$ then $\k$ remains measurable in $V[G]$.
\end{lemma}

The following Theorem is a consequence of the previously established results, the factorization property of the forcing motion $\mathbb{P}_{\bar{E}, \epsilon}$, and the application of some reflective arguments.

\begin{lemma}
Assume  $G$ is $\PP_{\bar{E}, \epsilon}$-generic over $V$. Let $\langle \k_\xi: \xi<\mu \rangle$ be an increasing enumeration of $C^{\bar{\k}}$, and for each $\xi<\mu,$ let $\lambda_\xi$ be the least inaccessible above $\k_\xi.$ Then
\begin{enumerate}
\item A cardinal $\eta<\k$ is collapsed in $V[G]$ iff there exists a limit ordinal
 $\xi<\mu,$ such that $\eta \in (\k_\xi, \gl_\xi),$ and then $\eta$ is collapsed to $\k_\xi,$

\item For each limit $\xi<\mu,$ $(\k_\xi^+)^{V[G]}= \gl_\xi.$
\end{enumerate}
\end{lemma}



Now we bring some definitions in a formal way.
\begin{defn}
Assume $M$ is an inner model, and $X$ is a class. Then
\begin{itemize}
    \item $\HOD_X$ denotes the class of all sets which are hereditarily ordinal-definable using parameters from $X$,
    \item $\HOD^M$ denotes the class of hereditarily ordinal-definable sets in the sense of model $M$.
    \item In general and if $X$ is a class of $M$,  we  define $(\HOD_X)^M$ as the class of hereditarily ordinal-definable sets with parameters from $X$ in the sense of model $M$.
\end{itemize}
\end{defn}

Note that, in particular, $\HOD$ is just $\HOD_X$,  where $X$ is the class of all ordinals.

\begin{defn}
    Assume $\mathbb{P}$ is a forcing notion.
    \begin{enumerate}
        \item $\mathbb{P}$ is called homogeneous, if for all
        $p, q \in \mathbb{P}$, there is an automorphism $\pi$
        of $\mathbb{P}$ with $\pi(p)=q.$

        \item $\mathbb{P}$ is called weakly homogeneous if and only if for every
two conditions $p, q$ in $\mathbb{P}$, there are $p'\leq p$
and  $q'\leq q$ and an automorphism $\pi: \mathbb{P}/ p' \rightarrow \mathbb{P} / q'$.
    \end{enumerate}
\end{defn}
It is clear that every homogeneous forcing notion is weakly homogeneous.
We will use the following well-known Fact.
\begin{lemma}
\label{weak}
 Assume $\MPB$ is a weakly  homogeneous ordinal definable forcing notion and let $G$ be $\MPB$-generic over $V$. Then $(\HOD_V)^{V[G]} \subseteq V.$
 \end{lemma}

 In the next Lemma we present a simple proof of Leshem's result stated above, using a measurable cardinal.

\begin{lemma}
\label{leshem result}
Assume $\kappa$ is regular and $\lambda>\kappa$ is a measurable cardinal. Then  $\Vdash_{\mathrm{Col}(\kappa, < \lambda)}$``$\lambda=\kappa^+ + \DTP(\lambda)$''.
\end{lemma}

\begin{proof}
Let $\MPB=\mathrm{Col}(\kappa, < \lambda)$ and let $G$ be $\MPB$-generic over $V$. Also let $U$ be a normal measure on $\lambda$
and let $j: V \to M \simeq \Ult(V, U)$ be the corresponding ultrapower embedding. Let $H$ be $\mathrm{Col}(\kappa, [\lambda, j(\lambda)))$-generic over $V[G]$. By standard arguments, in $V[G \times H],$ we can lift $j$ to some
elementary embedding $j: V[G] \to M[G \times H]$.

Now suppose that $T \in V[G]$ is a $\lambda$-tree which is definable in $H(\lambda)^{V[G]}=H(\lambda)[G]$. Then in $M[G \times H], T$ has a branch which is defined
in a natural way: take some $y \in j(T)_\lambda$ and let
\[
b=\left\{x \in j(T) \mid x <_{j(T)} y    \right\}=\left\{x \in T \mid x <_{j(T)} y    \right\}.
\]
But as the tree is definable and the forcing is homogeneous, we can easily see that $b$ is in fact in $V[G]$: let the formula $\phi$ and  $z \in H(\lambda)^{V[G]}$
be such that
\[
\alpha <_T \beta \iff H(\lambda)^{V[G]} \models \phi(z, \alpha, \beta).
\]
By the chain condition of the forcing, $z$ has a name $\dot{z} \in H(\lambda)^V,$ and so $j(\dot{z})=\dot{z}$. As $j$ is an elementary embedding, we have
\[
\alpha <_{j(T)} \beta \iff H(j(\lambda))^{M[G \times H]} \models \phi(z, \alpha, \beta).
\]
Note that $z, T \in M[G[,$ hence it follows that
$b \in (\HOD_{M[G]})^{M[G \times H]}$, and by the homogeneity of the forcing and Lemma \ref{weak}, $b\in M[G] \subseteq V[G]$.
\end{proof}

The following Lemma has been established in \cite{daghighi-pourmahdian}. However, for the sake of comprehensiveness, we include the proof here.
\begin{lemma}
\label{preservation of DTP}
Assume  $\DTP(\kappa^+)$ holds and let $\MPB$ be a weakly homogeneous forcing notion which preserves $H(\kappa^+).$ Then $\DTP(\kappa^+)$
holds in the generic extension by $\MPB.$
\end{lemma}

\begin{proof}
Assume $G$ is $\MPB$-generic over $V$ and let $T \in V[G]$ be a $\kappa^+$-tree definable in $H(\kappa^+)^{V[G]}$.
As $H(\kappa^+)^{V[G]}=H(\kappa^+)^{V}$,  $T$ is defined with parameters from $V$ and since the forcing is homogeneous, $T \in V.$ By our assumption,
$T$ has a branch in $V$ and hence in $V[G]$.
\end{proof}

\section{Proof of the Main Theorem}
\label{proof of main theorem}
In this section, we present the proof of the Main Theorem \ref{main theorem}. Assume that $\kappa$
is a supercompact cardinal and let $\lambda$ be the least measurable cardinal above $\kappa$. Let $\bar{E}= \langle E_\xi \mid \xi < \lambda  \rangle$
be a Mitchell increasing sequence of extenders
such that  for each $\xi < \lambda,\, \crit(j_\xi)=\kappa, ~^{<\lambda}M_\xi \subseteq M_\xi$  and $M_\xi \supseteq V_{\kappa+2}$,
where $j_\xi: V \to M_\xi \simeq \Ult(V, E_\xi)$ is the corresponding elementary embedding. Let $\MPB_{\bar{E}}$ be the supercompact extender based Radin forcing
using $\bar{E}$, and let $G$ be  $\MPB_{\bar{E}}$-generic over $V$. Let us recall the basic properties of  $\MPB_{\bar{E}}$.

\begin{lemma}
\label{basic properties}
The following hold in $V[G]$:
\begin{itemize}
\item [(a)] There exists a club $C= \langle \kappa_\xi \mid \xi < \kappa    \rangle$ of $\kappa$ consisting of $V$-measurable cardinals.
\item [(b)] For each limit ordinal $\xi < \kappa$ let $\lambda_\xi$ be the least measurable cardinal above $\kappa_\xi.$ Then $\lambda_\xi = \kappa_\xi^+$.
\item [(c)] $\kappa$ remains inaccessible.
\end{itemize}
\end{lemma}

\begin{lemma}
In $V[G],$ the definable tree property holds at all $\lambda_\xi,$ where $\xi < \kappa$ is a limit ordinal.
\end{lemma}

\begin{proof}
Assume $\xi < \kappa$ is a limit ordinal. Let $p \in G$ be of the form $p= p_0 ^{\frown} p_1,$ where $p_0 \in \MPB^*_{\bar{e}}$ with $\kappa(\bar{e})=\kappa_\xi$ and $p_1 \in \MPB^*_{\bar{E}}$. So we can factor $\MPB_{\bar{E}}/ p$ as
$$\MPB_{\bar{E}}/p =  \MPB_{\bar{e}}/ p_0 \times  \MPB_{\bar{E}}/ p_1,$$
where $ \MPB_{\bar{e}}/ p_0$ is essentially the forcing up to level
$\kappa_\xi$ below $p_0$ and $\MPB_{\bar{E}}/ p_1$ does not add any new subsets to $\lambda_\xi^+$.
Thus, by Lemma \ref{preservation of DTP}, it suffices to show that $\DTP(\lambda_\xi)$ holds in the generic extension by $ \MPB_{\bar{e}}/ p_0 $.
This follows by essentially the same ideas as in the proof of Main Theorem 3 from \cite{daghighi-pourmahdian} and the weak homogeneity
of the forcing notion $ \MPB_{\bar{e}}/ p_0 $ as proved in \cite{gitik-merimovich}.

The forcing $ \MPB_{\bar{e}}/ p_0$  is just homogeneous modulo Radin forcing, as we describe it below, hence more work is needed to get the result,
and so we sketch the proof for completeness.

Given $p= (f^p, T^p) \in \MPB^*_{\bar{E}}$, set $s(p)=(f^p \upharpoonright \left\{\kappa\right\}, T^p \upharpoonright \left\{ \kappa\right\})$, and by recursion, define the projection
of an arbitrary condition $p= p_0 ^{\frown} \dots ^{\frown} p_n \in \MPB_{\bar{E}}$, by  $s(p)= s(p_0 ^{\frown} \dots p_{n-1})^{\frown} s(p_n).$
Set $\MPB^\pi_{\bar{E}}=s``(\MPB_{\bar{E}})$. Then  $\MPB^\pi_{\bar{E}}$ is  the ordinary Radin forcing using the measure sequence $u= \langle E_\xi(\kappa) \mid  \xi < \lambda \rangle.$
As shown in \cite{gitik-merimovich} (see also \cite{golshani}), the forcing $\MPB_{\bar{E}} / H$ is weakly homogeneous, where $H=s``(G)$ is $\MPB^\pi_{\bar{E}}$-generic over $V$.

In $V$, $\lambda_\xi$ is a measureable cardinal, so let $k: V \to N$ witness this.
It follows  that $k(s)$ is a projection from  $k(\MPB_{\bar{E}})$ onto $k(\MPB^\pi_{\bar{E}})$.

Now let $T \in V[G]$ be a $\lambda_\xi$-tree, which is definable in
$H(\lambda_\xi)^{V[G]}=H(\lambda_\xi)^{N[G]}$. We have a natural projection
\[
\sigma: k(\MPB_{\bar{E}}) \longrightarrow \MPB_{\bar{E}}
\]
which  induces a  projection
\[
\sigma^\pi: k(\MPB^\pi_{\bar{E}}) \longrightarrow \MPB^\pi_{\bar{E}}
\]
so that the following diagram  commutes

\begin{center}
\begin{tikzcd}
k(\MPB_{\bar{E}}) \arrow[rr, "\sigma"] \arrow[dd, "k(s)"']& &\MPB_{\bar{E}} \arrow[dd, "s"]\\
&  & \\
 k(\MPB^\pi_{\bar{E}}) \arrow[rr, "\sigma^\pi"]   & & \MPB^\pi_{\bar{E}}
\end{tikzcd}
\end{center}

i.e., $\sigma^\pi \circ k(s) = s \circ \sigma.$
Let $K$ be $k(\MPB_{\bar{E}})$-generic over $V$ such that $\sigma``(K)=G.$ Then $L= k(s)``(K)$ is $k(\MPB^\pi_{\bar{E}}) $-generic over $V$
and $\sigma^\pi``(L)=H.$

It is easily seen that we can lift $k$ to some elementary embedding $k: V[G] \to N[K]$, which is defined
in $V[K].$ As in the proof of Lemma \ref{leshem result}, $T$ has a branch $b \in N[K]$ of the form
\[
b= \left\{ x \in k(T) \mid x <_{k(T)} y  \right\}=\left\{ x \in T \mid x <_{k(T)} y  \right\},
\]
where $y \in k(T)_{\lambda_\xi}$ is a node on $\lambda_\xi$-th level of $k(T)$. Assume $\phi$ defines $T$, so that for some $z \in H(\lambda_\xi)^{V[G]}$
\[
\alpha <_T \beta \iff H(\lambda_\xi)^{V[G]} \models \phi(z, \alpha, \beta).
\]
But $H(\lambda_\xi)^{V[G]}=H(\lambda_\xi)[G],$ so for some $\dot{z} \in H(\lambda_\xi), z=\dot{z}[G]$. We can assume from the start that
each element of $H(\lambda_\xi)$ is definable in $H(\lambda_\xi)$ using ordinal parameters less that $\lambda_\xi$, so without loss of generality, assume $z$ is an ordinal
less than $\lambda_\xi$. It follows that $k(z)=z$, and
\[
b=\left\{ x \in T \mid   H(k(\lambda_\xi))^{N[K]} \models \phi(z, x, y)    \right\} \in \HOD^{N[K]}.
\]
Note that $k$ is constant on $H(\lambda_\xi)$ and $\MPB^\pi_{\bar{E}}$ is essentially the Radin forcing at $\kappa_\xi < \lambda_\xi$ using the measure sequence $u$, so we can easily show that $k(\MPB^\pi_{\bar{E}})/H$ is weakly homogeneous. On the other hand, using the elementarity
of $k$ and by \cite{gitik-merimovich}, $k(\MPB_{\bar{E}}) / L$ is also homogeneous.
It easily follows that  $k(\MPB_{\bar{E}}) / H$ is indeed homogeneous, so $\HOD^{N[K]} \subseteq N[H]$  and hence
\[
\HOD^{N[K]} \subseteq N[H] \subseteq V[G].
\]
It follows that $b \in V[G],$ as required.
\end{proof}

From now on, we assume that  $\kappa_0=\aleph_0$ and that each limit point of $C$ is a singular cardinal in $V[G]$.

Now we work in $V[G]$, and let
\[
\MQB = \langle \langle \MQB_\xi \mid \xi \leq \kappa    \rangle, \langle  \dot{\MRB}_\xi \mid \xi < \kappa     \rangle\rangle
\]
 be the reverse Easton iteration of forcing notions where for each $\xi < \kappa,$ $\Vdash_{\MQB_\xi}$``$\dot{\MRB}_\xi=\dot{\C}ol(\kappa^*_\xi, < \kappa_{\xi+1})$'',
where $\kappa^*_\xi=\kappa_\xi^+$ if $\xi>0$ is a limit ordinal and $\kappa^*_\xi=\kappa_\xi$ otherwise. Let $H$ be $\MQB$-generic over $V[G].$

By Lemmas \ref{leshem result} and \ref{preservation of DTP},
\begin{center}
$V[G][H] \models$``$\DTP(\nu^+)$ holds for all regular cardinals $\nu < \kappa$''.
\end{center}

Note that in $V[G][H]$, there are no inaccessible cardinals below $\kappa$ (by our assumption on $C$)
and limit cardinals of $V[G][H]$ below $\kappa$  are of the form $\kappa_\xi,$ for some limit ordinal $\xi < \lambda$.
By Lemma \ref{basic properties}, $(\kappa_\xi^+)^{V[G][H]}=\lambda_\xi$, so the following Lemma completes the proof.

\begin{lemma}
$V[G][H] \models$``$\DTP(\lambda_\xi)$ holds for all limit ordinals $\xi < \kappa$''.
\end{lemma}

\begin{proof}
Let $\MQB=\MQB_\xi * \dot{\MQB}_\infty$, where $\MQB_\xi$ is the iteration up to level $\xi$
and $\Vdash_{\MQB_\xi}$``$\dot{\MQB}_\infty$ is $\lambda_\xi$-closed and homogeneous''.
So by Lemma \ref{preservation of DTP}, it suffices to show that
$\DTP(\lambda_\xi)$ holds in the extension by $\MPB_{\bar{E}}* \dot{\MQB}_\xi$.  As before, let $p \in G$ be of the form $p= p_0 ^{\frown} p_1,$ where $p_0 \in \MPB^*_{\bar{e}}$ with $\kappa(\bar{e})=\kappa_\xi$ and $p_1 \in \MPB^*_{\bar{E}}$ and  factor  $\MPB_{\bar{E}}/ p$ as $\MPB_{\bar{E}}/p =  \MPB_{\bar{e}}/ p_0 \times  \MPB_{\bar{E}}/ p_1.$

As $\MPB_{\bar{E}}/ p_1$ does not add new subsets to $\lambda_\xi^+$, the forcing notion $\MQB_\xi$ is computed in both models $V^{\MPB_{\bar{E}}/p}$
and $V^{\MPB_{\bar{e}}/p_0}$ in the same way,
so it suffices to prove the following:
\[
V^{\MPB_{\bar{e}}/p_0 * \dot{\MQB}_\xi} \models \text{~``~} \DTP(\lambda_\xi) \text{~holds''~}.
\]
The proof is essentially the same as before using the weak homogeneity of the corresponding forcing notions.
\end{proof}

\end{document}